\documentclass[reqno,12pt]{amsart}
\usepackage{amsmath,epsfig}
\usepackage{amssymb}

\newtheoremstyle{rem}{1.3ex}{1.3ex}{\rmfamily}{}
{\itshape\rmfamily}{}{1.5ex}{}

\newtheorem{theorem}{Theorem}[section]
\newtheorem{lemma}[theorem]{Lemma}

\theoremstyle{definition}

\newtheorem{remark}[theorem] {Remark}

\renewcommand{\section}{\secdef\sct\sect}
\newcommand{\sct}[2][default]{\refstepcounter{section}
\setcounter{equation}{0}
\vspace{0.5cm}
\centerline{ \large
\scshape \arabic{section}.\ #1}
\vspace{0.3cm}}
\newcommand{\sect}[1]{
\vspace{0.5cm}
\centerline{\large\scshape #1}
\vspace{0.3cm}}

\renewcommand{\subsection}{\secdef \subsct\sbsect}
\newcommand{\subsct}[2][default]{\refstepcounter{subsection}
\nopagebreak
\vspace{0.5\baselineskip}
{\flushleft\bf \arabic{section}.\arabic{subsection}~\bf #1  }
\nopagebreak}
\newcommand{\sbsect}[1]{\vspace{0.1cm}\noindent
{\bf #1}\vspace{0.1cm}}

\def\phi{\varphi }

\newcommand{\C}     {\mathbb{C}}

\newcommand{\R}     {\mathbb{R}}

\newcommand{\N}     {\mathbb{N}}

\newcommand{\E}     {\mathbb{E}}

\newcommand{\V}     {\mathbb{V}}

\def\1{{\mathchoice {1\mskip-4mu\mathrm l}
                    {1\mskip-4mu\mathrm l}
                    {1\mskip-4.5mu\mathrm l} {1\mskip-5mu\mathrm l}}}

\begin{document}

\title[Moderate deviations for the determinant]{\large
Moderate deviations for\\\vspace{2mm}the determinant\\\vspace{5mm}of Wigner matrices}

\author[Hanna D\"oring and Peter Eichelsbacher]{} 
\maketitle

\thispagestyle{empty}
\vspace{0.2cm}

\centerline{\sc Hanna D\"oring\footnote{Ruhr-Universit\"at Bochum, Fakult\"at f\"ur Mathematik,
NA 4/68, D-44780 Bochum, Germany, {\tt hanna.doering@ruhr-uni-bochum.de}}, Peter Eichelsbacher\footnote{Ruhr-Universit\"at Bochum, Fakult\"at f\"ur Mathematik,
NA 3/66, D-44780 Bochum, Germany, {\tt peter.eichelsbacher@ruhr-uni-bochum.de}
\\The second author has been supported by Deutsche Forschungsgemeinschaft via SFB/TR 12.}} 

\vspace{0.5cm}
\centerline{{\it Dedicated to Friedrich G\"otze on the occasion of his sixtieth birthday}}

\vspace{2 cm}

\begin{quote}
{\small {\bf Abstract:} }
We establish a moderate deviations principle (MDP) for the log-determinant $\log | \det (M_n) |$ of a Wigner matrix $M_n$ matching four moments with either the GUE or GOE ensemble. Further we establish Cram\'er--type moderate deviations and Berry-Esseen bounds for the log-determinant for the GUE and GOE
ensembles as well as for non-symmetric and non-Hermitian Gaussian random matrices (Ginibre ensembles), respectively. 
\end{quote}



\newpage
\setcounter{section}{0}

\section{Introduction}
In random matrix theory, the {\it determinant} is naturally an important functional. The study of determinants of random matrices has a long history.
The earlier papers focused on the determinant $\det A_n$ of a non-Hermitian iid matrix $A_n$, where the entries of the matrix were independent random variables with mean 0 and variance 1. Szekeres and Tur\'an \cite{SzekeresTuran:1937} studied an extremal problem. Later, in a series of papers moments
of the determinants were computed, see \cite{Prekopa:1967} and \cite{Dembo:1989} and references therein. In \cite{TaoVu:2006}, Tao and Vu proved for Bernoulli random matrices, that with probability tending to one as $n$ tends to infinity
\begin{equation} \label{LLNiidAn}
\sqrt{n!} \exp(-c \sqrt{n \log n}) \leq | \det A_n | \leq \sqrt{n!} \, \, \omega(n)
\end{equation}
for any function $\omega(n)$ tending to infinity with $n$. This shows that almost surely, $\log | \det A_n|$ is $( \frac 12 + o(1)) n \, \log n$.
In \cite{Goodman:1963}, Goodman considered the random Gaussian case, where the entries of $A_n$ are iid standard real Gaussian variables. Here the square of the determinant can be expressed as a product of independent chi-square variables and it was proved that
\begin{equation} \label{CLTiidAn}
\frac{ \log (|\det A_n|) - \frac 12 \log n! + \frac 12 \log n }{ \sqrt{ \frac 12 \log n}} \to N(0,1)_{\R},
\end{equation}
where $N(0,1)_{\R}$ denotes the real standard Gaussian (convergence in distribution). A similar analysis also works for complex Gaussian matrices, in which
the entries remain jointly independent but now have the distribution of the complex Gaussian $N(0,1)_{\C}$. In this case
a slightly different law holds true:
\begin{equation} \label{CLTiidAnC}
\frac{ \log (|\det A_n|) - \frac 12 \log n! + \frac 14 \log n }{ \sqrt{ \frac 14 \log n}} \to N(0,1)_{\R}.
\end{equation}
Girko  \cite{Girko:1979} stated that \eqref{CLTiidAn} holds for real iid matrices under the assumption that the fourth moment of the atom variables is 3. In \cite{Girko:1998}
he claimed the same result under the assumption that the atom variables have bounded $(4 + \delta)$-th moment. 
Recently, Nguyen and Vu \cite{NguyenVu:2011} gave a proof for \eqref{CLTiidAn} under an exponential decay hypothesis on the entries. They also present an estimate for the rate of convergence, which is that
the Kolmogorov distance of the distribution of the left hand side of \eqref{CLTiidAn} and the standard real Gaussian can be bounded by $\log^{- \frac 13 +o(1)} n$. In our paper we will be able to improve the bound to $\log^{-\frac 12} n$ in the Gaussian case. 

In the non-Hermitian iid model $A_n$ it is a crucial fact that the rows of the matrix are jointly independent.  This independence no longer holds true for {\it Hermitian} random matrices, which makes the analysis of determinants of Hermitian random matrices more challenging. The analogue of \eqref{LLNiidAn}
for Hermitian random matrices was first proved in \cite[Theorem 31]{Tao/Vu:2009} as a consequence of the famous Four Moment Theorem. Even in the Gaussian case, it is not simple to prove an analogue of the Central Limit Theorem (CLT) \eqref{CLTiidAnC}. The observations in \cite{Goodman:1963} do not apply due to the dependence between the rows. In \cite{MehtaNormand:1998} and in \cite{DelannayLeCaer:2000}, the authors computed the moment generating function of the log-determinant for the Gaussian unitary and Gaussian orthogonal ensembles, respectively,  and discussed the central limit theorem via the method of cumulants (see \cite[equation (40) and Appendix D]{DelannayLeCaer:2000}):
consider a Hermitian $n \times n$ matrix $X_n$ in which the atom distribution $\zeta_{ij}$ are given by the complex Gaussian $N(0,1)_{\C}$ for $i < j$ and the real Gaussian $N(0,1)_{\R}$ for $i=j$ (which is called the Gaussian Unitary Ensemble (GUE)). The calculations in  \cite{DelannayLeCaer:2000} 
should imply a Central Limit Theorem (see Remark \ref{RemarkDC} in our paper):
\begin{equation} \label{CLTiidMnC}
\frac{ \log (|\det X_n|) - \frac 12 \log n! + \frac 14 \log n }{ \sqrt{ \frac 12 \log n}} \to N(0,1)_{\R},
\end{equation}
Recently, Tao and Vu \cite{TaoVu:2011} presented a different approach to prove this result approximating the $\log$-determinant as a sum of weakly dependent
terms, based on analyzing a tridiagonal form of the GUE due to Trotter \cite{Trotter:1984}. They have to apply stochastic calculus and a martingale central limit theorem to get their result. This method is quite different and also quite involved. More important for us, the techniques due to Tao and Vu seem not to be applicable to get finer asymptotics  like Cram\'er--type moderate deviations, Berry-Esseen bounds and moderate deviations principles. The reason for this is 
the quality of the approximation by a sum of weakly dependent terms they have chosen is not sharp enough. Let us emphasize that Tao and Vu
proved the CLT \eqref{CLTiidMnC} for certain Wigner matrices,  generating a Four Moment Theorem for determinants.

The aim of our paper is to use a closed formula for the moments of the determinant of a GUE matrix, giving at the same time a
closed formula for the cumulant generating function of the log-determinant. We will be able to present good bounds for all cumulants.
As a consequence we will obtain Cram\'er--type moderate deviations, Berry-Esseen bounds and moderate deviation principle (for definitions see Section 2)
for the log-determinant of the GUE, improving results in \cite{DelannayLeCaer:2000} and \cite{TaoVu:2011}. Moreover we will obtain similar results for the GOE ensemble. 
Good estimates on the cumulants imply such results. To do so we apply a celebrated lemma of the theory of large deviations probabilities due to Rudzkis, Saulis and Statulevi{\v{c}}ius \cite{RSS:1978}, \cite{SaulisStratulyavichus:1989} as well as results on moderate deviation principles via cumulants due to the authors \cite{DoeringEichelsbacher:2010}. 
Applying the recent Four Moment theorem for determinants due to Tao and Vu \cite{TaoVu:2011}, we are able to prove the moderate deviation principle
and Berry-Esseen bounds for the log-determinant for Wigner matrices matching four moments with either the GUE or GOE ensemble. Moreover 
we will be able to prove moderate deviations results and will improve the Berry-Esseen type bounds in \cite{NguyenVu:2011} in the cases of non-symmetric and non-Hermitian Gaussian random matrices, called Ginibre ensembles.

Remark that the first universal result of a moderate deviations principle was proved in \cite{DoeringEichelsbacher:2011} and  \cite{DoeringEichelsbacher:2012} for the number of eigenvalues of a Wigner matrix, based on fine asymptotics of the variance of the
eigenvalue counting function of GUE matrices, on the Four Moment theorem and on localization results.

\section{Gaussian ensembles and Wigner matrices}

Among the ensembles of $n \times n$ random matrices $X_n$, Gaussian orthogonal and unitary ensembles have been studied extensively and are still being investigated.
Their probability densities are proportional to $\exp ( - \operatorname{tr}(X_n^2))$, where $\operatorname{tr}$ denotes the trace. Matrices are real symmetric
for the Gaussian orthogonal ensemble (GOE) and  Hermitian for the Gaussian unitary ensemble (GUE).
The joint distributions of eigenvalues for the Gaussian ensembles are (\cite[Theorem 2.5.2]{Zeitounibook}, \cite[Chapter 3]{Mehta:RandomMatrices})
\begin{equation} \label{jointdensityGE} 
P_{n, \beta} (\lambda_1, \ldots, \lambda_n) := \frac{1}{Z_{n, \beta}} \exp \biggl( - \frac{\beta}{4} \sum_{i=1}^{n} \lambda_i^2 \biggr) \prod_{1 \leq j < k \leq n} |\lambda_j - \lambda_k |^{\beta},
 \end{equation}
where $\beta = 1,2$ for the orthogonal and unitary ensembles, respectively, and $Z_{n, \beta}$ is the normalizing constant, sometimes called the Mehta integral (see \cite[Theorem 2.5.2, formula (2.5.4), and Corollay 2.5.9, Selberg's integral formula]{Zeitounibook}).

Let us denote by $X_n^{\beta}$ the random matrices of the two Gaussian ensembles.
We are interested in the moments of $| \det X_n^{\beta}| $ for these ensembles, that is
$$
M_{n, \beta}(s) := \langle | \det X_n^{\beta} |^s \rangle_{\beta} : = \int_{\R^n} P_{n, \beta} (\lambda_1, \ldots, \lambda_n) \prod_{i=1}^n | \lambda_i| ^s \, d \lambda_i.
$$
All information about the distribution of $\log | \det X_n^{\beta}|$ can be obtained from the generating function $M_{n, \beta}(s)$. The moments of
$\log | \det X_n^{\beta}|$ may be obtained from the coefficients in the Taylor expansion of $M_{n, \beta}$ evaluated at $s=0$,
$$
M_{n, \beta}(s)= \sum_{j \geq 0} \frac{ \langle  (\log | \det X_n^{\beta} |)^j \rangle_{\beta}}{j!} \, s^j,
$$
the corresponding cumulants $ \Gamma_j(n, \beta):=(-i)^j \frac{d^j}{dt^j} \log \E\bigl[e^{i t \log | \det X_n^{\beta}|}\bigr] \bigr|_{t=0}$ 
are related to the Taylor coefficients of $\log  M_{n, \beta}$ via
$$
\log M_{n, \beta}(s) = \sum_{j \geq 0} \frac{\Gamma_j(n, \beta)}{j!} \, s^j.
$$
In the literature the {\it Mellin transform} of the probability density of $|\det X_n^{\beta}|$ was calculated for the Gaussian ensembles, giving an
explicit  formula for $M_{n, \beta}(s)$. To be more precise, if $g_{n, \beta}(\cdot)$ denotes the probability density of the determinant of a GOE or a GUE matrix
and $g^+_{n, \beta}(y) := \frac 12 (g_{n, \beta}(y) + g_{n, \beta}(-y))$ be the even part, the Mellin transform of $g^+_{n, \beta}$ is defined by
$$
{\mathcal M}_{n, \beta}(s) := \int_0^{\infty} y^{s-1} g^+_{n, \beta}(y) dy.
$$
For the GOE and GUE ensembles we obtain
$$
{\mathcal M}_{n, \beta}(s)= \frac 12 \int_{-\infty}^{\infty} \cdots \int_{-\infty}^{\infty} P_{n, \beta}(\lambda_1, \ldots, \lambda_n) |\lambda_1 \cdots \lambda_n|^{s-1}
d \lambda_1 \cdots d\lambda_n
$$
and an obvious consequence is the relation 
\begin{equation} \label{MellinMoment}
M_{n, \beta}(s) = 2 {\mathcal M}_{n, \beta}(s+1).
\end{equation}
It is quite involved to calculate the Mellin transform even for the Gaussian ensembles. 
The case $\beta =1$ was calculated in \cite[formulas (31),(19) and (26)]{DelannayLeCaer:2000} (see also \cite[Chapter 26.5]{Mehta:RandomMatrices}).
Here the Mellin transform is a Pfaffian of an anti-symmetric matrix applying the method of (skew) orthogonal polynomials. With \eqref{MellinMoment}, for 
$n = 2p +1$ one obtains
\begin{equation} \label{moment1}
M_{2p+1, 1}(s) = 4^{n s/2} \prod_{m=1}^n \frac{\Gamma( \frac{s}{2} + \frac 12 + b_m^1)}{\Gamma( \frac 12 + b_m^1)}
\end{equation}
with $b_m^1 := \frac 12 \lfloor \frac{m-1}{2} \rfloor + \frac 14$. If $n=2p$ one obtains
\begin{equation} \label{moment1b}
M_{2p, 1}(s) = 2^{\frac{(n+1)s}{2}} F\biggl(\frac{s+1}{2}, - \frac{s}{2}; \frac{n+1+s}{2}; \frac 12\biggr)  \frac{\Gamma((s+1)/2) \Gamma((n+1)/2)}{\Gamma(\frac 12) 
\Gamma((n+1+s)/2)} \prod_{m=1}^p \frac{\Gamma( s+m + \frac 12)}{\Gamma( m +\frac 12)},
\end{equation}
where $F$ is the (Gau\ss{}) hypergeometric function
\begin{equation} \label{hypergeometric}
F( a,b;c ; z) := \sum_{m=0}^{\infty} \frac{(a)^{(m)} (b)^{(m)}}{(c)^{(m)}} \frac{z^m}{m!}
\end{equation}
with $(x)^{(m)} :=x(x+1)(x+2) \cdots (x+m-1)$ denoting the Pochhammer symbol. $F$ is convergent for arbitrary $a,b,c$ and for real $-1 < z < 1$. 
In \cite{AndrewsGouldenJackson:2003}, an alternative derivation for \eqref{moment1} and \eqref{moment1b} is presented using terminating hypergeometric series.
The case $\beta=2$ was calculated in \cite[Section 2]{MehtaNormand:1998}. Here a knowledge of determinants and orthogonal polynomials is needed.
One obtains 
\begin{equation} \label{moment2}
M_{n, 2}(s) = 2^{n s/2} \prod_{m=1}^n \frac{\Gamma( \frac{s}{2} + \frac 12 + b_m^2)}{\Gamma( \frac 12 + b_m^2)}
\end{equation}
with $b_m^2 = \lfloor \frac m2 \rfloor$. 
As a consequence of \eqref{moment2} we obtain the following results for the cumulants $\Gamma_j(n,2)$  of $\log | \det X_n^{2} |$:

\begin{lemma}[Bounds for the cumulants of $\log |\det X_n^2|$, GUE] \label{cumulantsGUE}
For the Gaussian unitary ensemble $\beta=2$ we obtain 
$$
\Gamma_1(n,2) = - \frac n2 ( 1 + \log 2) + \frac n2 \log \bigl( 2 \lfloor n/2 \rfloor \bigr) + \operatorname{const} + O(1/n)
$$
and
$$
\sigma_2^2 := \Gamma_2(n,2)= \frac 12 \log \bigl( 2 \lfloor n/2 \rfloor \bigr)
 + \frac 12 (\gamma +  \log 2 +1) + O(1/n),
$$
where $\gamma$ denotes the Euler-Mascheroni constant.
Moreover for any $j \geq 3$ we have
\begin{equation} \label{cumbound1}
\big| \Gamma_j(n,2) \big| \leq \operatorname{const} \, j! .
\end{equation}
\end{lemma}

\begin{proof}
Let us remark that some of our calculations can be found in \cite{DelannayLeCaer:2000}. We work out all the details to get good bounds 
on the cumulants, which is not the aim in \cite{DelannayLeCaer:2000}.
With $\psi(x) := \frac{d}{dx} \log \Gamma(x)$ we denote the {\it digamma function}. From $\eqref{moment2}$ we obtain
\begin{equation} \label{f1}
\Gamma_1(n,2) = \frac{d}{ds} \log M_{n,2}(s) \bigg|_{s=0} = \frac n2 \log 2 + \frac 12 \sum_{i=1}^n \psi( 1/2 + b_i^2).
\end{equation}
For any $n=2k+1$ we obtain $ \frac 12 \sum_{i=1}^n \psi(1/2 + b_i^2) =  \sum_{j=1}^k \psi(1/2 + j) + \frac 12 \psi(\frac 12)$ and 
for $n=2k$ we have $ \frac 12 \sum_{i=1}^n \psi(1/2 + b_i^2) =  \sum_{j=1}^k \psi(1/2 + j) + \frac 12 \psi(1/2) - \frac 12 \psi( \frac{n+1}{2})$.
With $\Gamma(1+x) = x \Gamma(x)$ it follows that $\psi(1+x) = \psi(x) + \frac 1x$ and therefore recursively
$
\psi(1/2+j) = \psi(1/2) + 2 \biggl( \sum_{l=1}^j \frac{1}{2l-1} \biggr)$, see \cite[Section 1.3, (1.3.9)]{Lebedev:1965}.
Using
$$
2 \sum_{j=1}^k \sum_{l=1}^j \frac{1}{2l-1} = 2(k+1) \sum_{l=1}^k \frac{1}{2l-1} - \sum_{l=1}^k \frac{2l}{2l-1} = (2k+1) \biggl( \sum_{l=1}^{2k} \frac 1l - \sum_{l=1}^k \frac{1}{2l} \biggr) - k
$$
we obtain
$
\sum_{j=1}^k \psi(1/2 + j) = k \psi(1/2) -k + (2k+1) \biggl( \sum_{l=1}^{2k} \frac 1l - \sum_{l=1}^k \frac{1}{2l} \biggr)$.
Applying
\begin{equation} \label{harmonic}
\sum_{l=1}^n \frac 1l = \gamma + \log n + \frac{1}{2n} +O(\frac{1}{n^2}),
\end{equation}
it follows that
$
(2k+1) \biggl( \sum_{l=1}^{2k} \frac 1l - \sum_{l=1}^k \frac{1}{2l} \biggr)= (2k +1) \frac 12 (\gamma + 2 \log 2) + (2k +1) \frac 12 \log k + O(\frac 1k)$. 
With $\psi(1/2) = -2 \log 2 - \gamma$ we have
\begin{equation} \label{spaeter}
\sum_{j=1}^k \psi(1/2 + j) + \frac 12 \psi(1/2) = -k + (k + \frac 12) \log k + O(\frac 1k).
\end{equation}
In the case $n=2k$ we have to consider in addition the term $\frac 12 \psi(1/2 + k) = \frac 12 \log k + O( \frac 1k)$.
Summarizing we obtain for every $n$:
$$
\Gamma_1(n,2) =  - \frac n2 \bigl( \log 2 + 1 \bigr) + \frac n2 \log (2k) + \operatorname{const} + O(1/n).
$$
From $\eqref{moment2}$ and $\eqref{f1}$ we obtain for $n=2k+1$
\begin{equation} \label{f4}
\Gamma_j(n,2) = \frac{d^j}{ds^j} \log M_{n,2}(s) \bigg|_{s=0} = \frac{1}{2^j} \psi^{(j-1)}(1/2) + \frac{1}{2^{j-1}} \sum_{i=1}^k \psi^{(j-1)}(1/2 + i)
\end{equation}
with the {\it polygamma function} $\psi^{(k)} (x) := \frac{d^k}{dx^k} \log \Gamma(x)$. For $n=2k$ one has to subtract from the right hand side the term $\frac{1}{2^j} \psi^{(j-1)}(\frac{n+1}{2})$. We remind  the representation of $\Gamma(x)^{-1}$ due to Weierstrass (see for example \cite[Section 1.3, (1.3.17)]{Lebedev:1965}):
$
\frac{1}{\Gamma(x)} = x e^{\gamma x} \prod_{k=1}^{\infty} (1 + \frac xk) e^{-\frac xk}$.
Differentiating $- \log \Gamma(x)$ leads to
$$
\psi(x) = - \gamma - \frac 1x + \sum_{k=1}^{\infty} \biggl( \frac 1k - \frac{1}{x+k} \biggr) = - \gamma + \sum_{n=0}^{\infty} \biggl( \frac{1}{n+1} - \frac{1}{x+n} \biggr).
$$
Therefore one obtains
\begin{equation} \label{f3}
 \psi^{(k)}(x) = (-1)^{k+1} \, k! \, \sum_{n=0}^{\infty} \frac{1}{(x+n)^{k+1}}.
\end{equation}
It follows that
\begin{eqnarray*} \label{f2}
\sum_{i=1}^k \psi^{(j-1)}(1/2 +i) &  = & (-1)^{j} \, (j-1)! \, 2^j  \sum_{i=1}^k \sum_{m=i}^{\infty} \frac{1}{(2m+1)^{j}} \\
& & \hspace{-3cm} = (-1)^{j} \, (j-1)! \, 2^{j-1} \biggl( 2 \sum_{i=1}^k \sum_{m=i}^{k} \frac{1}{(2m+1)^{j}}  + 2 \sum_{i=1}^k \sum_{m=k+1}^{\infty} \frac{1}{(2m+1)^{j}} \biggr)  =: T_1 + T_2.
\end{eqnarray*}
With
$
2 \sum_{i=1}^k \sum_{m=i}^k \frac{1}{(2m+1)^{j}}  = \sum_{m=1}^k \frac{1}{(2m+1)^{j-1}} - \sum_{m=1}^k \frac{1}{(2m+1)^{j}}
$
we obtain
$$
T_1 = (-1)^{j} \, (j-1)! \, 2^{j-1}  \sum_{m=0}^k \frac{1}{(2m+1)^{j-1}} - (-1)^{j} \, (j-1)! \, 2^{j-1} -(-1)^{j} \, (j-1)! \, 2^{j-1} \sum_{m=1}^k \frac{1}{(2m+1)^{j}}.
$$
Further we get
$$
T_2= (-1)^{j} \, (j-1)! \, 2^{j-1}  \, 2k \,  \sum_{m=k+1}^{\infty} \frac{1}{(2m+1)^{j}}.
$$
Hence using \eqref{f3} for $\psi^{(j-1)}$ we obtain
\begin{eqnarray} \label{jallgemein}
\sum_{i=1}^k \psi^{(j-1)}(1/2 +i)
& = & (-1)^{j} \, (j-1)! \, 2^{j-1} \sum_{m=0}^{k} \frac{1}{(2m+1)^{j-1}} -\frac{1}{2}\psi^{(j-1)}\bigl(\frac{1}{2}\bigr) \nonumber \\
&   & {}+ (-1)^{j} \, (j-1)! \, 2^{j-1} (2k+1) \sum_{m=k+1}^{\infty} \frac{1}{(2m+1)^{j}}.
\end{eqnarray}
In particular for $j=2$, we have
\begin{eqnarray}\label{jgleichzwei}
\sum_{i=1}^k \psi^{(1)}(1/2 +i) 
&=& 2 \bigl(\frac{1}{2}\log(k)+\frac{1}{2}(\gamma+2\log(2))\bigr)-\frac{1}{2} \psi^{(1)}(1/2) + \frac{1}{2} (2k+1) \psi^{(1)}\bigl(k+\frac{3}{2}\bigr) \nonumber \\
&=& \log(k)+\gamma+2\log(2)-\frac{1}{2} \psi^{(1)}(1/2) +1+O\bigl(\frac{1}{n}\bigr). 
\end{eqnarray}
With \eqref{f4} we obtain for $n = 2k+1$ that
$$
\Gamma_j(n,2) = (-1)^j (j-1)! \sum_{m=0}^k \frac{1}{(2m+1)^{j-1}} + (-1)^j (j-1)! \, (2k+1) \, \sum_{m=k+1}^{\infty} \frac{1}{(2m+1)^{j}}.
$$
The first term is $- 2^{1-j} (j-1) \psi^{(j-2)}(\frac 12) + O(1/k)$. The second term is
$2^{-j} (2k+1)  \psi^{(j-1)}(\frac 12 + k +1)$. 
For $n=2k$ we have to subtract $2^{-j} \psi^{(j-1)}(\frac 12 +k)$.
Finally we will apply some bounds for the polygamma functions $\psi^{(j)}$. Therefore we will apply the following integral-representation (see for example
\cite[Section 1.4, (1.4.12)]{Lebedev:1965}):
\begin{equation} \label{integral1}
\psi(x) = \log(x) - \int_0^{\infty} e^{-tx} \biggl( t f(t) + \frac 12 \biggr) \, dt \quad \text{with} \quad f(t) := \biggl( \frac 12 - \frac 1t + \frac{1}{e^t -1} \biggr) \, \frac 1t, \,\, t \geq 0.
\end{equation}
Differentiating we see that for $j \geq 1$:
\begin{equation} \label{integral2}
\psi^{(j)}(x) = (-1)^{j-1} j! x^{-j} + (-1)^{j-1}  \int_0^{\infty} e^{-tx} t^j \biggl( t f(t) + \frac 12 \biggr) \, dt.
\end{equation}
Notice that $0 < \bigl( t f(t) + \frac 12 \bigr) < 1$ for every $t \geq 0$; hence we obtain for every $x \geq 0$ and every $j \geq 1$:
\begin{equation} \label{polygammabound}
|\psi^{(j)}(x)| \leq j! x^{-j} + j! x^{-j-1}.
\end{equation}
Let us consider the variance $\sigma_2^2 = \Gamma_2(n,2)$. With \eqref{polygammabound} we have $|\psi^{(1)}(1/2 + k)| \leq (\frac 12 +k)^{-1} + (\frac 12 +k)^{-2}$. Hence we  have
$\sigma_2^2 =  \frac 12 \sum_{i=1}^k \psi^{(1)}(1/2 +i)  + \frac 12 \psi(1/2) + O(1/k)$ and with \eqref{jgleichzwei}
we obtain
$$
\sigma_2^2 = \frac 12  \log k + \frac 12 (\gamma + 2 \log 2 +1) + O(1/k).
 $$
For $j \geq 3$ the cumulants can be bounded by: With \eqref{polygammabound} we obtain
\begin{eqnarray*}
|\Gamma_j(n,2)| & \leq & \bigg| 2^{1-j} (j-1) \psi^{(j-2)}(1/2) \bigg| + \bigg| 2^{-j} (2k+1)  \psi^{(j-1)}(1/2 + k +1) \bigg| \\
& + & \bigg| 2^{-j} \psi^{(j-1)}(1/2 +k) \bigg|+ O(1/k) \\
& \leq & 6 (j-1)! + \operatorname{const} \, \biggl( \frac{(j-1)!}{2^{j-1}} \, \frac{1}{k^{j-2}} + \frac{(j-1)!}{2^{j-1}} \, \frac{1}{k^{j-1}} \biggr) 
\leq   \operatorname{const}  (j-1)!.
\end{eqnarray*}
Therefore the cumulants satisfy the stated bounds.
\end{proof}
With some more technical effort we obtain similar results for the Gaussian orthogonal ensembles:
 \begin{lemma}[Bounds for the cumulants of $\log |\det X_n^1|$, GOE] \label{cumulantsGOE}
For the orthogonal Gaussian ensemble ($\beta = 1$) we obtain
$$
\Gamma_1(n,1) = \frac n2 \log  \bigl( 2 \lfloor n/2 \rfloor \bigr) - \frac n2 + \operatorname{const} + O(1/n)
$$
and
$$
\sigma_1^2 := \Gamma_2(n,1)= \log \bigl( 2 \lfloor n/2 \rfloor \bigr) + \frac{\gamma}{2} +1 -2K + \frac{\pi^2}{4}   + O(1/n), 
$$
where $K$ denotes Catalan's constant $K = \sum_{m=0}^{\infty} \frac{(-1)^m}{(2m+1)^2}$,
and for any $j \geq 3$
$$
\big| \Gamma_j(n,1 ) \big| \leq  \operatorname{const} \, j! .
$$
\end{lemma}

\begin{proof}
For $\beta=1$ and $n=2k+1$, formula \eqref{moment1} for the Mellin transform implies
\begin{eqnarray*}
\Gamma_1(n,1)
&=& \frac{d}{ds} \log M_{n,s}(s)\big|_{s=0}
= \frac{n}{2}\log(4) + \frac{1}{2} \sum_{i=1}^n \psi\bigl(\frac{1}{2} + \frac{1}{2}\bigl\lfloor\frac{i-1}{2}\bigr\rfloor +\frac{1}{4}\bigr)
\\
&=& n\log(2) + \sum_{i=0}^{k-1} \psi\bigl(\frac{3}{4}+\frac{i}{2}\bigr) +\frac{1}{2} \psi\bigl(\frac{3}{4}+\frac{k}{2}\bigr)
\\
&=& n\log(2) + \frac{1}{2}\psi\bigl(\frac{3}{4}\bigr)
+ \sum_{i=1}^{k} \biggl(\frac{1}{2}\psi\bigl(\frac{3}{4}+\frac{i-1}{2}\bigr) +\frac{1}{2} \psi\bigl(\frac{3}{4}+\frac{i}{2}\bigr)\biggr).
\end{eqnarray*}
The last transformation is useful since we are now able to apply Legendre's duplication formula $\Gamma(z) \Gamma(z + 1/2) = 2^{1-2z} \sqrt{\pi} \Gamma(2z)$ (see for example \cite[Section 1.2]{Lebedev:1965}). This implies 
\begin{equation}\label{legendre}
\frac{1}{2} \psi(z) + \frac{1}{2} \psi\bigl(z+\frac{1}{2}\bigr)=\psi(2z) -\log(2).
\end{equation}
With $z = 3/4 + i/2 -1/2$ we obtain 
\begin{equation}\label{gamma_1(n,1)}
\Gamma_1(n,1)
= n\log(2) + \frac{1}{2}\psi\bigl(\frac{3}{4}\bigr)
+\sum_{i=1}^{k} \psi\bigl(1/2+i\bigr) - k\log(2).
\end{equation}
The summand $\frac{1}{2}\psi\bigl(\frac{3}{4}\bigr)$ equals via the same identity $\psi\bigl(\frac{1}{2}\bigr)-\log(2)-\frac{1}{2}\psi\bigl(\frac{1}{4}\bigr)= \psi\bigl(\frac{1}{2}\bigr)-\log(2)+\frac{\pi}{4}+\frac{3}{2}\log(2)+\frac{1}{2}\gamma = \frac{\pi}{4}-\frac{3}{2}\log(2)-\frac{1}{2}\gamma$.
As in the GUE case, we have $ \sum_{i=1}^{k} \psi\bigl(1/2+i\bigr)=  -\frac{1}{2}\psi\bigl(\frac{1}{2}\bigr)- k + \bigl(k+\frac{1}{2}\bigr)\log(k)+O\bigl(\frac{1}{k}\bigr)$, see \eqref{spaeter}. Now \eqref{gamma_1(n,1)} implies that
$$
\Gamma_1(n,1)= \frac{n}{2}\log(n-1) -\frac{n}{2} +\frac{\pi+2}{4} +O\bigl(\frac{1}{n}\bigr).
$$
The $j$th cumulant, $j\geq 2$, is given by
\begin{eqnarray*}
\Gamma_j(n,1)
&=& \frac{d^j}{ds^j} \log M_{n,s}(s)\big|_{s=0}
= \frac{1}{2^j} \sum_{i=1}^n \psi^{(j-1)}\bigl(\frac{1}{2} + \frac{1}{2}\bigl\lfloor\frac{i-1}{2}\bigr\rfloor +\frac{1}{4}\bigr)
\\
&=& \frac{1}{2^{j-1}} \sum_{i=0}^{k-1} \psi^{(j-1)}\bigl(\frac{3}{4}+\frac{i}{2}\bigr) +\frac{1}{2^j} \psi^{(j-1)}\bigl(\frac{3}{4}+\frac{k}{2}\bigr).
\end{eqnarray*}
Differentiating \eqref{legendre} implies
$
\psi^{(j-1)}(2z) = \frac{1}{2^j} \psi^{(j-1)}(z) + \frac{1}{2^j} \psi^{(j-1)}\bigl(z+\frac{1}{2}\bigr)
$
and therefore
\begin{equation}\label{gamma_j(n,1)}
\Gamma_j(n,1)
= \frac{1}{2^j} \psi^{(j-1)}\bigl(\frac{3}{4}\bigr) + \sum_{i=1}^{k} \psi^{(j-1)}\bigl(1/2+i \bigr)
\end{equation}
hold.
The duplicity formula for $z=\frac{1}{4}$ implies 
$
\frac{1}{4}\psi^{(1)}\bigl(\frac{3}{4}\bigr)
=\psi^{(1)}\bigl(\frac{1}{2}\bigr)-\frac{1}{4}\psi^{(1)}\bigl(\frac{1}{4}\bigr),
$
where
$\psi^{(1)}\bigl(\frac{1}{4}\bigr)
= 16 \sum_{m=0}^{\infty} \frac{1}{(4m+1)^2}
= 8 \sum_{m=0}^{\infty} \bigl(\frac{1}{(2m+1)^2}+\frac{(-1)^m}{(2m+1)^2} \bigr)
= 2 \sum_{m=0}^{\infty} \frac{1}{(m+\frac{1}{2})^2}+8\sum_{m=0}^{\infty}\frac{(-1)}{(2m+1)^2}
= 2 \psi^{(1)}\bigl(\frac{1}{2}\bigr) + 8K
$
with Catalan's constant $K$, resulting in
$\frac{1}{4}\psi^{(1)}\bigl(\frac{3}{4}\bigr)= \frac{\pi^2}{4}-2K$.
With \eqref{gamma_j(n,1)} and \eqref{jgleichzwei} we can conclude
\begin{equation}\label{gamma_2(n,1)}
\Gamma_2(n,1)=\frac{1}{4} \psi^{(1)}\bigl(\frac{3}{4}\bigr) + \sum_{i=1}^{k} \psi^{(1)}\bigl(1/2+i\bigr)= \frac{\pi^2}{4} -2K+
\log(k)+\frac{\gamma}{2}+\log(2) +1+O\bigl(\frac{1}{n}\bigr).
\end{equation}
For every $j\geq3$, the first summand can be bounded using \eqref{polygammabound}
$$
\bigl|\frac{1}{2^j}\psi^{(j-1)}\bigl(\frac{3}{4}\bigr)\bigr|
\leq (j-1)! \bigl(\frac{2}{3}\bigr)^{j-1} + (j-1)!  2 \bigl(\frac{2}{3}\bigr)^{j}
= (j-1)! \frac{7}{3} \bigl(\frac{2}{3}\bigr)^{j-1},
$$
and the remaining sum in \eqref{gamma_j(n,1)} is the same as in the GUE case:
With \eqref{jallgemein} we have
$$
\sum_{i=1}^k \psi^{(j-1)}\bigl(1/2 +i\bigr) + \frac{1}{2}\psi^{(j-1)}\bigl(\frac{1}{2}\bigr)
=  -2 (j-1) \psi^{(j-2)}\bigl(\frac{1}{2}\bigr)
	+ (2k+1) \psi^{(j-1)}\bigl(1/2+k+1\bigr)
	+O\bigl(\frac{1}{k}\bigr).
$$
Applying \eqref{polygammabound} we obtain
$ 
\biggl|\sum_{i=1}^k \psi^{(j-1)}(1/2 +i) + \frac{1}{2}\psi^{(j-1)}\bigl(\frac{1}{2}\bigr)\biggr|
\leq 
\operatorname{const}  (j-1)!,
$
which implies the bound for the $j$th cumulant, $j\geq3$. 

In the case of $n=2k$ even, we have to study the asymptotic behaviour of the hypergeometric function (see \eqref{hypergeometric}):
$
F\biggl(\frac{s+1}{2}, - \frac{s}{2}; \frac{n+1+s}{2}; \frac 12\biggr)
:=1+\sum_{m=1}^{\infty} x_m$,
denoting
${\displaystyle
\frac{\bigl(\frac{1+s}{2}\bigr)^{(m)} \bigl(-\frac{s}{2}\bigr)^{(m)}}
{\bigl(\frac{n+1+s}{2}\bigr)^{(m)}}
\frac{1}{2^m m!} \textrm{ by } x_m.
}$
Each $x_m$ is of order $O(n^{-m})$ and, for $s\in[0,2)$ and $n$ large enough, the hypergeometric function takes values in the interval $(-1,1)$. Therefore we can study the power series of the logarithm and get
$$
\log F\biggl(\frac{s+1}{2}, - \frac{s}{2}; \frac{n+1+s}{2}; \frac 12\biggr)
= \log\biggl(1+\sum_{m=1}^{\infty} x_m\biggr)
= \sum_{m=1}^{\infty} x_m 
+\sum_{l=2}^{\infty} (-1)^l \frac{1}{l} \bigl(\sum_{m=1}^{\infty} x_m\bigr)^l.
$$
We differentiate each $x_m$ via the quotient rule and the product rule in the enumerator. Setting $s=0$, the only remaining term in the enumerator is the one where we differentiate the factor $-\frac{s}{2}$. Thus the square of the denominator cancels out. The derivative of $x_m$ equals a constant times $\frac{1}{2^m m!} \frac{1}{\bigl(\frac{n+1}{2}\bigr)^{(m)}}$.
It follows that the sum over $l$ is of order $O(n^{-1})$, too.
Similarly we obtain that for every $j \geq 1$
$$\frac{d^j}{d s^j} \log F\biggl(\frac{s+1}{2}, - \frac{s}{2}; \frac{n+1+s}{2}; \frac 12\biggr) \bigg|_{s=0} = O \bigl( 1/n \bigr).
$$
Thus with \eqref{moment1b} and \eqref{spaeter} it follows that
\begin{eqnarray*}
\Gamma_1(n,1)&=&
\frac{n+1}{2} \log(2) + \frac{d}{ds}\log F\biggl(\frac{s+1}{2}, - \frac{s}{2}; \frac{n+1+s}{2}; \frac 12\biggr)\bigg|_{s=0} + \frac{1}{2} \psi\bigl(\frac 12\bigr) - \frac 12 \psi\bigl(\frac{n+1}{2}\bigr)
\\&&{}+\sum_{m=1}^k \psi\bigl(1/2+m \bigr)
\\
&=&
\frac{n+1}{2} \log(2) + O\bigl(\frac{1}{n}\bigr) + \frac{1}{2} \psi\bigl(\frac 12\bigr) - \frac 12 \psi\bigl(\frac{n+1}{2}\bigr)
-\frac{1}{2} \psi\bigl(\frac 12\bigr)-\frac{n}{2}+\frac{n+1}{2}\log\bigl(\frac{n}{2}\bigr)
\\
&=&
\frac{n}{2}\log(n) -\frac{n}{2} + \frac{1}{2}\log(2) +O\bigl(1/n \bigr)
\end{eqnarray*}
and by \eqref{jallgemein}
\begin{eqnarray*}
\Gamma_j(n,1)
&=& \frac{d^j}{ds^j}\log F\biggl(\frac{s+1}{2}, - \frac{s}{2}; \frac{n+1+s}{2}; \frac 12\biggr)\bigg|_{s=0} + \frac{1}{2^j} \psi^{(j-1)}\bigl(\frac 12\bigr) - \frac{1}{2^j} \psi^{(j-1)}\bigl(\frac{n+1}{2}\bigr)
\\
&&{}+\sum_{m=1}^k \psi^{(j-1)}\bigl(1/2+ m \bigr)
\\
&=& \frac{1}{2^j} \psi^{(j-1)}\bigl(\frac 12\bigr) - \frac{1}{2^j} \psi^{(j-1)}\bigl(\frac{n+1}{2}\bigr)+\sum_{m=1}^k \psi^{(j-1)}\bigl(1/2+m \bigr) +O\bigl(1/n\bigr).
\end{eqnarray*}
Note that the only difference to the case $n=2k+1$, see \eqref{gamma_j(n,1)}, is the summand $\frac{1}{2^j} \psi^{(j-1)}\bigl(\frac{n+1}{2}\bigr)$, which is of order $O\bigl(1/n\bigr)$. Therefore the second and higher cumulant satisfy the stated bounds for all $n$. \end{proof}

Good estimates on cumulants imply asymptotic results for the log-determinant of GUE and GOE ensembles, respectively. 
Before we state our results, we remind the reader on Cram\'er-type moderate deviations and a moderate deviation principle.
The classical result due to Cram\'er is the following. For independent and identically distributed random variables $X_1, \ldots, X_n$ with
$\E(X_1)=0$ and $\E(X_1^2) =1$ such that $\E e^{t_0 |X_1|} \leq c < \infty$ for some $t_0>0$, the following expansion for tail probabilities can be proved:
$$
\frac{ P(W_n > x)}{1 - \Phi(x)} = 1 + O(1) (1+x^3) / \sqrt{n}
$$
for $0 \leq x \leq n^{1/6}$ with $W_n := (X_1 + \cdots + X_n) / \sqrt{n}$, $\Phi$ the standard normal distribution function, and $O(1)$ depends on $c$ and $t_0$.
This result is sometimes called a {\it large deviations relation}. 
Let us recall the definition of a large deviation principle (LDP) due to Varadhan, see for example Dembo and Zeitouni
\cite{Dembo/Zeitouni:LargeDeviations}. A sequence of probability measures $\{(\mu_n), n\in \mathbb N\}$ on a
topological space $\mathcal X$ equipped with a $\sigma$-field $\mathcal
B$ is said to satisfy the LDP with speed $s_n\nearrow \infty$ and
good rate function $I(\cdot)$ if the level sets $\{x: I(x)\leq
\alpha\}$ are compact for all $\alpha\in[0,\infty)$ and for all
$\Gamma\in\mathcal B$ the lower bound
$$
\liminf_{n\to\infty}  \frac{1}{s_n} \log \mu_n(\Gamma)
\geq - \inf_{x\in \operatorname{int}(\Gamma)} I(x)
$$
and the upper bound
$$
\limsup_{n\to\infty}  \frac{1}{s_n} \log \mu_n(\Gamma)
\leq - \inf_{x\in \operatorname{cl}(\Gamma)} I(x)
$$
hold. Here $\operatorname{int}(\Gamma)$ and
$\operatorname{cl}(\Gamma)$ denote the interior and closure of
$\Gamma$ respectively. We say a sequence of random variables satisfies
the LDP when the sequence of measures induced by these variables
satisfies the LDP. Formally a moderate deviation principle is nothing
else but the LDP. However, we will speak about a moderate deviation 
principle (MDP) for a sequence of random variables, whenever the scaling
of the corresponding random variables is between that of an ordinary Law
of Large Numbers and that of a Central Limit Theorem. 

We consider
\begin{equation} \label{mainobject}
W_{n, \beta} := \frac{ \log | \det X_n^{\beta} | - \Gamma_1(n, \beta)}{\sigma_{\beta}} \quad \text{for} \quad \beta = 1,2
\end{equation}
as well as
\begin{equation} \label{mainobjectnumeric}
\widetilde{W}_{n, \beta} := \frac{ \log | \det X_n^{\beta} | - \frac n2 \log n + \frac n2 }{\sqrt{ \frac{1}{\beta} \log n}} \quad \text{for} \quad \beta = 1,2.
\end{equation}

\begin{theorem} \label{main}
For $\beta = 1,2$ we can prove:
(1) Cram\'er--type moderate deviations: \quad There exists two constants $C_1$ and $C_2$ depending on $\beta$, such that the following
inequalities hold true:
$$
\bigg| \log \frac{P(W_{n, \beta} \geq x)}{1 - \Phi(x)} \bigg| \leq C_2 \frac{1 + x^3}{\sigma_{\beta}}
$$
and
$$
\bigg| \log \frac{P(W_{n, \beta} \leq -x)}{\Phi(-x)} \bigg| \leq C_2 \frac{1 + x^3}{\sigma_{\beta}}
$$
for all $0 \leq x \leq C_1 \sigma_{\beta}$. On all cases $\sigma_{\beta}$ is of order $\sqrt{ \log n}$.

(2) Berry-Esseen bounds: \quad We obtain the following bounds:
$$
\sup_{x \in \R} \big| P(W_{n, \beta} \leq x) - \Phi(x) \big| \leq C(\beta) (\log n)^{-1/2},\quad \sup_{x \in \R} \big| P(\widetilde{W}_{n, \beta} \leq x) - \Phi(x) \big| \leq C(\beta) (\log n)^{-1/2}.
$$

(3) Moderate deviations principle: \quad For any sequence $(a_n)_n$ of real numbers such that
$1 \ll a_n \ll \sigma_{\beta}$
the sequences $\bigl( \frac{1}{a_n} W_{n, \beta} \bigr)_n$ and $\bigl( \frac{1}{a_n} \widetilde{W}_{n, \beta} \bigr)_n$ 
satisfy a MDP with speed $a_n^2$ and rate function $I(x) = \frac{x^2}{2}$, respectively.
\end{theorem}

\begin{remark} \label{RemarkDC}
The Berry-Esseen bound implies the Central Limit Theorem stated in \eqref{CLTiidMnC}. The statement of the central limit theorem in \cite{DelannayLeCaer:2000}
 was given differently. In section III, they considered a variance of order $2 \sigma^2 = \frac{1}{\beta n}$, meaning that the spectrum of the GUE model is concentrated on a finite interval (the support of the semicircular law). Then the $D$ is the determinant of the rescaled (!) GUE model, given a
 $\frac n2 \log n + n \log 2$ summand in addition to the expectation $-n( \frac 12 + \log 2) + O( \frac 1n)$ they stated in \cite[(43)]{DelannayLeCaer:2000}. 
 This is actually the expectation in \eqref{CLTiidMnC}. Choosing the variance $\sigma^2 = \frac{1}{4 n}$ in the case $\beta=2$
implies that we have to rescale each matrix-entry $\zeta_{ij}$ by $\zeta_{ij}/ (2 \sqrt{n})$ and hence the determinant of the rescaled matrix is $2^n n^{n/2}$
times the determinant of the matrix $X_n^2$.
\end{remark}

\begin{proof}
With the bound on the cumulants \eqref{cumbound1} we obtain that
$\big| \Gamma_j (W_{n,2}) \big| \leq 7 \frac{j!}{\sigma_2^j}$.
With $\sigma_2^2 \geq \frac 12 (\gamma + 2 \log 2 +1)$ we get
$$
\big| \Gamma_j (W_{n,2}) \big|
\leq j! \frac{1}{\sigma_2^{j-2}} \frac{7\cdot 2}{(\gamma + 2 \log 2 +1)}
\leq j! \frac{1}{\sigma_2^{j-2}} 5
\leq j! \Bigl(\frac{5}{\sigma_2}\Bigr)^{j-2}
\leq \frac{j!}{\Delta^{j-2}}
$$
with $\Delta = \sigma_2/5$ for all $n\geq 2$.
With Lemma 2.3 in \cite{SS:1991} one obtains 
$$
\frac{P\bigl( W_{n,2} \geq x\bigr) }{1 - \Phi(x)} = \exp(L(x)) \biggl( 1 + q_1 \phi(x) \frac{x+1}{\Delta_1} \biggr)
$$ and
$$
\frac{P \bigl( W_{n,2} \leq -x \bigr) }{\Phi(-x)} = \exp(L(-x)) \biggl( 1 + q_2 \phi(x) \frac{x+1}{\sqrt{2}\Delta_1} \biggr)
$$
for $0 \leq x \leq \Delta_1$, where $\Delta_1 = \sqrt{2} \Delta / 36$, 
$$
\phi(x) =\frac{60 \bigl(1 + 10 \Delta_1^2 \exp\bigl( -(1 -x/\Delta_1) \sqrt{\Delta_1} \bigr) \bigr)}{1 - x / \Delta_1},
$$ 
$q_1, q_2$ are constants in the interval $[-1,1]$ and $L$ is a function defined in \cite[Lemma 2.3, eq. (2.8)]{SS:1991} satisfying
$
\big| L(x) \big| \leq \frac{|x|^3}{3 \Delta_1} \quad \text{for all} \, \, x \,\, \text{with} \,\, |x| \leq \Delta_1$.
The Cram\'er--type moderate deviations follow applying \cite[Lemma 6.2]{DoeringEichelsbacher:2011}. The Berry-Esseen bound follows from
\cite[Lemma 2.1]{SS:1991} which is
$$
\sup_{x \in \R} \big| P\bigl( W_{n,2} \leq x \bigr) - \Phi(x) \big| \leq \frac{18}{\Delta_1} = \operatorname{const} \frac{1}{(\log n)^{1/2}}.
$$
The same Berry-Esseen bound follows using the asymptotic behavior of the first two moments.
Finally the MDP follows from \cite[Theorem 1.1]{DoeringEichelsbacher:2010} 
which is a MDP for $\bigl( \frac{1}{a_n} W_{n,2} \bigr)_n$ for any
sequence $(a_n)_n$ of real numbers growing to infinity slow enough such that $a_n / \Delta \to 0$ as $n \to \infty$. Moreover
$\bigl( \frac{1}{a_n} W_{n,2} \bigr)_n$ and $\bigl( \frac{1}{a_n} \widetilde{W}_{n,2} \bigr)_n$ are exponentially equivalent in the sense of
\cite[Definition 4.2.10]{Dembo/Zeitouni:LargeDeviations}: with $\hat{W}_{n,2} := \frac{ \log | \det X_n^{2} | - \frac n2 \log n + \frac n2 }{\sigma_2}$
we have that $|W_{n,2} - \hat{W}_{n,2}| \to 0$ as $n \to \infty$, and it follows that $\bigl( \frac{1}{a_n} \hat{W}_{n,2} \bigr)_n$ and  
$\bigl( \frac{1}{a_n} W_{n,2} \bigr)_n$ are exponentially equivalent. By Taylor we have
$\big| \frac{1}{a_n} (\hat{W}_{n,2} - \widetilde{W}_{n,2}) \big| = o(1) \,  \hat{W}_{n,2}$ and hence the result follows with 
\cite[Theorem 4.2.13]{Dembo/Zeitouni:LargeDeviations}.
\end{proof}

Next we will consider the following class of random matrices. Consider two independent families of i.i.d. random variables $(Z_{i,j})_{1 \leq i <j}$ (complex-valued) and $(Y_i)_{1 \leq i}$ (real-valued), zero mean, such that
$\E Z_{1,2}^2=0, \E|Z_{1,2}|^2=1$ and $\E Y_1^2=1$. Consider the (Hermitian) $n \times n$ matrix $M_n$ with entries $M_n^*(j,i) = M_n(i,j) = Z_{i,j}$ 
for $i <j $ and  $M_n^*(i,i) = M_n(i,i)=Y_i$. Such a matrix is called {\it Hermitian Wigner matrix}. 
The GUE matrices are the special case with complex Gaussian random variables $N(0,1)_{\C}$ in the upper triangular and 
real Gaussian random variables $N(0,1)_{\R}$ on the diagonal. 

We say that a Wigner Hermitian matrix obeys Condition $(C_1)$ for some constant $C$ if one has
\begin{equation} \label{condC1}
\E |Z_{i,j}|^C \leq C_1 \quad \text{and} \quad \E |Y_{i}|^C \leq C_1 
\end{equation}
for some constant $C_1$ independent on $n$. Two Wigner Hermitian matrices $M_n=(\zeta_{i,j})_{1 \leq i ,j \leq n}$ and $M_n' =(\zeta_{i,j}')_{1 \leq i ,j \leq n}$ match to order $m$ off the diagonal and to order $k$ on the diagonal if one has
$$
\E ( (\operatorname{Re}(\zeta_{i,j}))^a (\operatorname{Im}(\zeta_{i,j}))^b ) = \E ( (\operatorname{Re}(\zeta_{i,j}'))^a (\operatorname{Im}(\zeta_{i,j}'))^b ) 
$$
for all $1 \leq i\leq j \leq n$ and natural numbers $a,b \geq 0$ with $a+b \leq m$ for $i<j$ and
$a+b \leq k$ for $i=j$.

Applying \cite[Theorem 5]{TaoVu:2011}, the Four Moment Theorem for the determinant, we are able to prove
an MDP for the log-determinant even for a class of Wigner Hermitian matrices. For any Wigner Hermitian matrix $M_n$
consider
$$
W_n := \frac{ \log | \det M_n | - \frac 12 \log n! + \frac 14 \log n}{\sqrt{\frac 12 \log n}}.
$$

\begin{theorem}[Universal moderate deviations principle] \label{TH2}
Let $M_n$ be a Wigner Hermitian matrix whose atom distributions are independent of $n$, have real and imaginary parts that are independent and match
GUE to fourth order and obey Condition $(C_1)$, \eqref{condC1}, for some sufficiently large $C$, then for any sequence $(a_n)_n$ of real numbers such that
$1 \ll a_n \ll \sqrt{\log n}$ the sequence $\bigl( \frac{1}{a_n}  W_n \bigr)_n$
satisfies a MDP with speed $a_n^2$ and rate function $I(x) = \frac{x^2}{2}$. If $M_n$ matches GOE instead of GUE, then one instead has that
$ \bigl( \frac{\sqrt{  \frac 12 \log n}}{a_n \sqrt{\log n}} W_n \bigr)_n$ satisfies the MDP with same speed and rate function.
\end{theorem}

\begin{proof}
Let $M_n$ be the Wigner Hermitian matrix whose entries satisfy the conditions of the theorem and $M_n'$ denotes the GUE matrix.  Then \cite[Theorem 5]{TaoVu:2011} says that there
exists a small $c_0 >0$ such that for all $G : \R \to \R_+$ with $\big| \frac{d^j}{dx^j} G(x) \big| = O(n^{c_0})$ for $j=0, \ldots, 5$, we have
$$
\big| \E \bigl(G( \log | \det(M_n) |)\bigr) -  \E \bigl(G( \log | \det(M_n') |)\bigr) \big| \leq n^{-c_0}
$$
We consider for any $b, c \in \R$ the interval $I_n := [b_n, c_n]$ with 
$$
b_n := b \,a_n \sqrt{ \frac 12 \log n} + \frac 12 \log n! - \frac 14 \log n \,\, \text{and} \,\, c_n := c \, a_n \sqrt{ \frac 12 \log n} + \frac 12 \log n! - \frac 14 \log n
$$
With $I_n^+ := [b_n - n^{-c_0/10},  c_n + n^{-c_0/10} ]$ and $I_n^- := [b_n + n^{-c_0/10},  c_n - n^{-c_0/10} ]$ we construct a bump function $G_n: \R \to \R_+$
which is equal to one on the smaller interval $I_n^-$ and vanishes outside the larger interval $I_n^+$. It follows that $P( \log | \det(M_n) | \in I_n) \leq \E G_n(\log | \det(M_n)|)$ and  $\E G_n(\log | \det(M_n')|) \leq P( \log | \det(M_n') | \in I_n^+)$. One can choose $G_n$ to satisfy the condition 
$\big| \frac{d^j}{dx^j} G_n(x) \big| = O(n^{c_0})$ for $j=0, \ldots, 5$ and hence
\begin{equation} \label{intern1}
P( \log | \det(M_n) | \in I_n) \leq P( \log | \det(M_n') | \in I_n^+) + n^{-c_0}.
\end{equation}
By the same argument we get
\begin{equation} \label{intern2}
P( \log | \det(M_n') | \in I_n^-) - n^{-c_0} \leq P( \log | \det(M_n) | \in I_n).
\end{equation}
With $P \bigl( \frac{1}{a_n} W_n \in [b,c] \bigr) = P \bigl( \log | \det(M_n) |\in I_n \bigr)$. With \eqref{intern1} and \cite[Lemma 1.2.15]{Dembo/Zeitouni:LargeDeviations} we see that
$$
\limsup_{n \to \infty} \frac{1}{a_n^2} \log P \bigl( W_n/a_n \in [b,c] \bigr) \leq \max \biggl( \limsup_{n \to \infty} \frac{1}{a_n^2} \log 
P( \log | \det(M_n') | \in I_n^+); \limsup_{n \to \infty} \frac{1}{a_n^2} \log n^{-c_0} \biggr).
$$
For the first object we have
$$
\limsup_{n \to \infty} \frac{1}{a_n^2} \log 
P( \log | \det(M_n') | \in I_n^+) = \limsup_{n \to \infty} \frac{1}{a_n^2} \log P \biggl( \frac{1}{a_n} \widetilde{W}_{n,2} \in [b - \eta(n), c + \eta(n)] \biggr)
$$
with $\eta(n) := n^{-c_0/10} \bigl( a_n \sqrt{ \frac 12 \log n} \bigr)^{-1} \to 0$ as $n \to \infty$. Since $c_0 >0$ and $\log n /a_n^2 \to \infty$ for $n \to \infty$
by assumption, applying Theorem \ref{main} we have
$$
\limsup_{n \to \infty} \frac{1}{a_n^2} \log P \bigl( W_n/a_n \in [b,c] \bigr) \leq - \inf_{x \in [b,c]} \frac{x^2}{2}.
$$
Applying \eqref{intern2} we obtain in the same manner that
$$
\limsup_{n \to \infty} \frac{1}{a_n^2} \log P \bigl( W_n/a_n \in [b,c] \bigr) \geq - \inf_{x \in [b,c]} \frac{x^2}{2}.
$$
The conclusion follows applying \cite[Theorem 4.1.11 and Lemma 1.2.18]{Dembo/Zeitouni:LargeDeviations}. 
\end{proof}

\begin{remark}
The bump function $G_n$ in the proof of Theorem \ref{TH2} can be chosen to fulfill $\big| \frac{d^j}{dx^j} G_n(x) \big| = O(n^{c_0})$ for $j=0, \ldots, 5$
uniformly in the endpoints of the interval $[b,c]$. Hence the Berry-Esseen bound in Theorem \ref{main} can be obtained for Wigner matrices considered
in Theorem \ref{TH2}: 
$$
\sup_{x \in \R} \big| P(W_n \leq x) - \Phi(x) \big| \leq \operatorname{const} \, \bigl( (\log n)^{-1/2} + n^{-c_0} \bigr).
$$
We omit the details.
\end{remark}

\section{Non-symmetric and non-Hermitian Gaussian random matrices}

As already mentioned, recently Nguyen and Vu proved in \cite{NguyenVu:2011}, that for $A_n$ be an $n \times n$ matrix whose entries are independent real random variables with mean zero and variance one, the Berry-Esseen bound
$$
\sup_{x \in \R} \big| P( W_n \leq x ) - \Phi(x) \big| \leq \log^{-1/3 + o(1)} n
$$ with
\begin{equation} \label{Wn}
W_n := \frac{ \log (| \det A_n|) - \frac 12 \log(n-1)!}{\sqrt{ \frac 12 \log n}}
\end{equation}
holds true. We will prove good bounds for the cumulants of $W_n$ in the case where the entries are Gaussian random variables. Therefore we will be able to prove Cram\'er--type moderate deviations and an MDP as well as a Berry-Esseen bound of order $(\log n)^{-1/2}$ (and it seems
that one cannot have a rate of convergence better than this).
In the Gaussian case, again the calculation of the Mellin transform is the main tool. Fortunately, the transform can be calculated much easier.

Let $A_n$ be an $n \times n$ matrix whose entries are independent real or complex Gaussian random variables with mean zero and variance one. 
Denote by $A_n^{\dag}$ the transpose or Hermitian conjugate of $A_n$ according as $A_n$ is real or complex. Then $A_n A_n^{\dag}$ is positive semi-definite and its eigenvalues are real and non-negative. The positive square roots of the eigenvalues of $A_n A_n^{\dag}$ are known as the singular values of $A_n$. One has that
$$
\prod_{i=1}^n \lambda_i^2 = \det (A_n A_n^{\dag}) = | \det A_n|^2 = \prod_{i=1}^n |x_i|^2,
$$
where $\lambda_i$ are the singular values and $x_i$ are the eigenvalues of $A_n$. Now $A_n A_n^{\dag}$ is called {\it Wishart matrix}. For the real
case we consider independent $N(0,1)_{\R}$ distributed entries, for the complex case we assume that the real and imaginary parts are independent
and $N(0,1)_{\R}$ distributed entries. These ensembles are called {\it Ginibre ensembles}. One obtains for the joint probability distribution of the eigenvalues of $A_n A_n^{\dag}$ on $\R^n_+$ the density
$$
\frac{1}{\tilde{Z}_{n, \beta}} \exp \bigl( - \frac{\beta}{2} \sum_{i=1}^n y_i \bigr) \prod_{i=1}^n y_i^{\beta/2 -1} \prod_{i<j} |y_i -y_j|^{\beta}
$$
with $\beta=1$ for the real and $\beta=2$ for the complex case and $\tilde{Z}_{n, \beta}$ being the normalizing constant
(see for example \cite[Chapter 7]{Anderson:1971}). As a result the Gaussian joint probability density for the singular values $\lambda_i$ gets transformed to
$$
Q_{n,\beta}(\lambda_1, \ldots, \lambda_n) := \frac{1}{Z_{n,\beta}(n)} \exp \bigl( - \frac{\beta}{2} \sum_{i=1}^n \lambda_i^2 \bigr) \prod_{i=1}^n \lambda_i^{\beta -1} 
\prod_{i<j} |\lambda_i^2 - \lambda_j^2|^{\beta}
$$
with
\begin{equation} \label{norm}
Z_{n, \beta}(p) := \int \cdots \int \exp \bigl( - \frac{\beta}{2} \sum_{i=1}^n \lambda_i^2 \bigr) \prod_{i=1}^n \lambda_i^{(p-n) + \beta -1} \prod_{i<j} |\lambda_i^2 -\lambda_j^2|^{\beta} \prod_{i=1}^n d \lambda_i
\end{equation}
Now the Mellin transform of the probability density of the determinant of $A_n$ is given by
$$
{\mathcal M}_{n,\beta}(s) = \int_0^{\infty} \cdots  \int_0^{\infty} |\lambda_1 \cdots \lambda_n|^{s-1} Q_{n,\beta}(\lambda_1, \ldots, \lambda_n) \prod_{i=1}^n d \lambda_i = \frac{Z_{n, \beta}(n+s-1)}{Z_{n, \beta}(n)}.
$$
But using the Selberg identity of the Laguerre form, \cite[formula 17.6.5]{Mehta:RandomMatrices}, we obtain for the moment generating function 
$M_{n,\beta}(s) = {\mathcal M}_{n, \beta}(s-1)$:
\begin{equation} \label{trick}
M_{n,\beta}(s)= \bigl( \frac{2}{\beta} \bigr)^{ns/2} \prod_{i=1}^n \frac{\Gamma\bigl( (s + i \, \beta)/2 \bigr)}{\Gamma\bigl( (i \, \beta)/2 \bigr)}.
\end{equation}
This formula makes even sense for $\beta =4$, where $A_n$ is a quaternion matrix and $A_n^{\dag}$ denotes the dual of $A_n$ (see \cite[Section 15.4]{Mehta:RandomMatrices} for a discussion of the definition of a determinant in this case). We will concentrate on the real case $\beta=1$. The results
of the following theorem can be stated and proved similarly in the two other cases $\beta=2,4$. We omit the details. We consider $W_n$ as in \eqref{Wn}
and
\begin{equation} \label{variante}
\widetilde{W}_{n} := \frac{ \log | \det A_n | - \E(\log | \det A_n |)}{\V(\log | \det A_n |)^{1/2}}.
\end{equation}

\begin{theorem} \label{main2}
Let $A_n$ be an $n \times n$ matrix whose entries are independent real $N(0,1)_{\R}$ random variables. Then we have:
(1) Cram\'er--type moderate deviations: \quad There exists two constants $C_1$ and $C_2$ depending on $\beta$, such that the following
inequalities hold true:
$$
\bigg| \log \frac{P(\widetilde{W}_{n} \geq x)}{1 - \Phi(x)} \bigg| \leq C_2 \frac{1 + x^3}{\sigma_{\beta}}
$$
and
$$
\bigg| \log \frac{P(\widetilde{W}_{n} \leq -x)}{\Phi(-x)} \bigg| \leq C_2 \frac{1 + x^3}{\sigma_{\beta}}
$$
for all $0 \leq x \leq C_1 \V(\log | \det A_n |)^{1/2}$. 

(2) Berry-Esseen bounds: \quad We obtain the following bounds:
$$
\sup_{x \in \R} \big| P(W_{n} \leq x) - \Phi(x) \big| \leq C(\beta) (\log n)^{-1/2},\quad \sup_{x \in \R} \big| P(\widetilde{W}_{n} \leq x) - \Phi(x) \big| \leq C(\beta) (\log n)^{-1/2}.
$$

(3) Moderate deviations principle: \quad For any sequence $(a_n)_n$ of real numbers such that
$1 \ll a_n \ll \sigma_{\beta}$
the sequences $\bigl( \frac{1}{a_n} W_{n} \bigr)_n$ and $\bigl( \frac{1}{a_n} \widetilde{W}_{n} \bigr)_n$ 
satisfies a MDP with speed $a_n^2$ and rate function $I(x) = \frac{x^2}{2}$, respectively.
\end{theorem}

\begin{proof}
With \eqref{trick} we are able to estimate the cumulants $\Gamma_j(n)$ of $\log | \det A_n |$. The calculations will benefit from 
a few results presented in the proofs of Lemma \ref{cumulantsGUE} and Lemma \ref{cumulantsGOE}. Therefore we restrict ourselves to the major steps of
the proof. We denote by $\psi$ the digamma function and by $\psi^{(k)}$, $k \in \N$, the polygamma function (see Lemma \ref{cumulantsGUE}).
With \eqref{trick} we have
$$
\Gamma_1(n) = \frac n2 \log n + \frac 12 \sum_{i=1}^n \psi(i/2) \quad \text{and} \quad \Gamma_j(n) = \frac{1}{2^j}  \sum_{i=1}^n \psi^{(j-1)}(i/2) \,\,  \text{for} \,\,
 j \geq 2.
$$
For $n=2k+1$ we have $\frac 12 \sum_{i=1}^n \psi(i/2) = \frac 12 \bigl(  \sum_{i=0}^k \psi(1/2+i) + \sum_{i=1}^k \psi(i) \bigr)$. Using \eqref{spaeter} the first summand is equal to $-\frac k2 + \frac k2 \log k + \frac 14 \log k + \frac 14 \psi(1/2) + O(1/k)$. With $\psi(1+x) = \psi(x) + \frac 1x$ (see Lemma \ref{cumulantsGUE}) one obtains that $\psi(i) = \psi(1) + \sum_{j=1}^{i-1} \frac 1j$. Thus applying \eqref{harmonic} we have $\frac 12 \sum_{i=1}^k \psi(i) =\frac k2 \log(k-1) - \frac k2 + \operatorname{const} + O(1/k)$. Summarizing we get
\begin{eqnarray*}
\Gamma_1(2k+1) & = & -k + k \log k + \frac 14 \log k + \operatorname{const} + O(1/k) \\
& = & - \frac n2 (1 + \log 2)  + \frac n2 \log (n-1) - \frac 14 \log (n-1)  + \operatorname{const} + O(1/n).
\end{eqnarray*}
Therefore the leading term of the expectation of $\log | \det A_n |$ is $\log\bigl( (n-1)! \bigr)$. In the case $n=2k$ one obtains the same order.
For $\Gamma_j(2k+1)$ with $j \geq 2$ we proceed as following:
$$
\Gamma_j(2k+1) = \frac{1}{2^j}  \sum_{i=1}^{2k+1} \psi^{(j-1)}(i/2) = \frac{1}{2^j} \biggl( \psi^{(j-1)}(1/2) + \sum_{i=1}^k \psi^{(j-1)}(1/2 + i) + \sum_{i=1}^k \psi^{(j-1)}(i) \biggr).
$$
Take the representation \eqref{f3} to see that $\psi^{(j-1)}(i) = (-1)^j(j-1)! \sum_{m=i}^{\infty} \frac{1}{m^j}$, such that
\begin{eqnarray*}
\sum_{i=1}^k \psi^{(j-1)}(i) & = & (-1)^j (j-1)! \biggl( \sum_{m=1}^k \frac{1}{m^{j-1}}  + k \sum_{m=k+1}^{\infty} \frac{1}{m^j} \biggr) \\
& = & -(j-1) \psi^{(j-2)}(1) + O(1/k) + k \psi^{(j-1)}(k+1).
\end{eqnarray*} 
With the help of \eqref{jallgemein} we obtain for $j \geq 3$ that 
\begin{eqnarray*}
\Gamma_j(n) = \frac{1}{2^{j+1}} \psi^{(j-1)}(1/2) &-& \frac{1}{2^j} (j-1) \bigl( \psi^{(j-2)}(1/2) + \psi^{(j-2)}(1) \bigr) \\ 
&+&  \frac{1}{2^{j+1}} (2k+1) \psi^{(j-1)}(1/2+k+1) + \frac{1}{2^{j}} \, k \, \psi^{(j-1)}(k+1) +O(1/k).
\end{eqnarray*} 
With \eqref{polygammabound} we are able to bound the cumulants in a similar way as in the proof of Lemma \ref{cumulantsGUE} and obtain
$|\Gamma_j(n)| \leq \operatorname{const} j!$. Moreover with \eqref{jgleichzwei} we obtain for the variance
$$
\Gamma_2(n) = \frac 12 \log n + \frac 12 \bigl( \gamma + 1 + \frac{\pi^2}{8} \bigr) + O(1/n).
$$
Therefore the leading term of the variance of $\log | \det A_n |$ is $\frac 12 \log n$. Now the theorem follows exactly as in the proof of Theorem \ref{main}. 
\end{proof}

\begin{remark}
Let $A_n$ be an $n \times n$ matrix whose entries are independent complex and quaternion, respectively. Then $W_n$ and $\widetilde{W}_{n}$ as defined before satisfy Cram\'er--type moderate deviations, Berry-Esseen bounds and a moderate deviations principle. This can easily be checked noting that, for $\beta=1,2,4$,
\vspace{-3ex}
$$
\Gamma_j^{(\beta)}(n)=\frac{n}{2}\log{\bigl(\frac{2}{\beta}\bigr)} \delta_{\{j=1\}}
+ \frac{1}{2^j}\sum_{i=1}^n \psi^{(j-1)}\biggl(\frac{i\beta}{2}\biggr)
$$
is of order $\frac{1}{2\beta}\log(n)$:
For $\beta=2$ we have already bounded these summands in the proof above. In the case $\beta=4$ use \eqref{legendre} and its derivatives to see, that the cumulant can be represented via sums of $\psi^{(j-1)}(i)$ and $\psi^{(j-1)}(i+1/2)$.
\end{remark}

\begin{remark}[Trace-fixed ensembles] 
In \cite{LeCaerDelannay:2003}, the authors considered fixed-trace Gaussian random matrix ensembles (real-symmetric and Hermitian ones).  
Here the trace of the matrix is kept constant with no other restriction on the matrix elements. These ensembles are shown to be equivalent as far
as finite moments of the matrix elements are concerned. Especially, the Mellin transform of the fixed-trace Gaussian matrices can be deduced from
the Mellin transform of the Gaussian orthogonal and unitary ensemble, respectively, see \cite[formulas (17), (20) and (22)]{LeCaerDelannay:2003}.
Hence it is expected that the distribution of the log-determinant of these ensembles is asymptotically Gaussian with a variance of order $\log n$.
We would be able to deduce the results in Theorem \ref{main2} for the Gaussian trace-fixed ensembles by the same technique. We omit
the details. Remark, that universal limits for the eigenvalue correlation functions in the bulk of the spectrum for fixed trace matrix ensembles
are considered in \cite{GoetzeGordin:2008}, \cite{GoetzeGordinLewina:2007}. In this case, the class of matrices are of nondeterminantal structure. 
\end{remark}


\providecommand{\MRhref}[2]{%
  \href{http://www.ams.org/mathscinet-getitem?mr=#1}{#2}
}
\providecommand{\href}[2]{#2}

\newcommand{\SortNoop}[1]{}\def\cprime{$'$} \def\cprime{$'$}
  \def\polhk#1{\setbox0=\hbox{#1}{\ooalign{\hidewidth
  \lower1.5ex\hbox{`}\hidewidth\crcr\unhbox0}}}
\providecommand{\bysame}{\leavevmode\hbox to3em{\hrulefill}\thinspace}
\providecommand{\MR}{\relax\ifhmode\unskip\space\fi MR }
\providecommand{\MRhref}[2]{%
  \href{http://www.ams.org/mathscinet-getitem?mr=#1}{#2}
}
\providecommand{\href}[2]{#2}

\end{document}